\documentclass[english]{article}
\usepackage[T1]{fontenc}
\usepackage[latin9]{inputenc}
\usepackage{amsthm}
\usepackage{amssymb}

\makeatletter

\providecommand{\tabularnewline}{\\}

\newcommand{\lyxaddress}[1]{
\par {\raggedright #1
\vspace{1.4em}
\noindent\par}
}
\theoremstyle{plain}
\newtheorem{thm}{\protect\theoremname}
\theoremstyle{plain}
\newtheorem{lem}[thm]{\protect\lemmaname}
\ifx\proof\undefined
\newenvironment{proof}[1][\protect\proofname]{\par
\normalfont\topsep6\p@\@plus6\p@\relax
\trivlist
\itemindent\parindent
\item[\hskip\labelsep
\scshape
#1]\ignorespaces
}{%
\endtrivlist\@endpefalse
}
\providecommand{\proofname}{Proof}
\fi

\makeatother

\usepackage{babel}
\providecommand{\lemmaname}{Lemma}
\providecommand{\theoremname}{Theorem}

\begin{document}

\title{Finding all squared integers expressible as the sum of consecutive
squared integers using generalized Pell equation solutions with Chebyshev
polynomials}

\author{Vladimir Pletser}

\maketitle

\lyxaddress{European Space Research and Technology Centre, ESA-ESTEC P.O. Box
299, NL-2200 AG Noordwijk, The Netherlands; E-mail: Vladimir.Pletser@esa.int}
\begin{abstract}
\noindent Square roots $s$ of sums of $M$ consecutive integer squares
starting from $a^{2}\geq1$ are integers if $M\equiv0,9,24$ or $33\left(mod\,72\right)$;
or $M\equiv1,2$ or $16\left(mod\,24\right)$; or $M\equiv11\left(mod\,12\right)$
and cannot be integers if $M\equiv3,5,6,7,8$ or $10\left(mod\,12\right)$.
Finding all solutions with $s$ integer requires to solve a Diophantine
quadratic equation in variables $a$ and $s$ with $M$ as a parameter.
If $M$ is not a square integer, the Diophantine quadratic equation
in variables $a$ and $s$ is transformed into a generalized Pell
equation whose form depends on the $M\left(mod\,4\right)$ congruent
value, and whose solutions, if existing, yield all the solutions in
$a$ and $s$ for a given value of $M$. Depending on whether this
generalized Pell equation admits one or several fundamental solution(s),
there are one or several infinite branches of solutions in $a$ and
$s$ that can be written simply in function of Chebyshev polynomials
evaluated at the fundamental solutions of the related simple Pell
equation. If $M$ is a square integer, it is known that $M\equiv1\left(mod\,24\right)$
and $M=\left(6n-1\right)^{2}$ for all integers $n$; then the Diophantine
quadratic equation in variables $a$ and $s$ reduces to a simple
difference of integer squares which yields a finite number of solutions
in $a$ and $s$ to the initial problem.
\end{abstract}


\section{Introduction}

Lucas' cannonball problem \cite{key-1-24,key-1-15} of finding a square
number of cannonballs stacked in a square pyramid has only two solutions,
$1$ and $4900$, the later corresponding to the sum of the first
$24$ squared integers and was proven by several authors \cite{key-1-16,key-1-17,key-1-18,key-1-21,key-1-19,key-1-20}. 

More generally, finding all integers $s$ equal to the sum of $M$
consecutive integer squares starting from $a^{2}\geq1$ involves solving
a single Diophantine quadratic equation in three variables, two independent
($a$ and $M$) and one dependent ($s$). Philipp \cite{key-1-23},
extending the previous work of Alfred \cite{key-1-22}, proved that
there are a finite or an infinite number of solutions depending on
whether $M$ is or not a square integer and in the later case, using
a form of the generalized Pell equation. Beeckmans \cite{key-1-4},
after demonstrating eight necessary conditions on $M$ with a table
of values of $M<1000$ and the smallest values of $a>0$, developed
a method based on solving generalized Pell equations to provide all
solutions. In two previous papers, this Author showed \cite{key-10-1}
that no solution exists if $M$ is congruent to $3,5,6,7,8$ or $10\left(mod\,12\right)$
using Beeckmans necessary conditions, and that integer solutions exist
if $M$ is congruent to $0,1,2,4,9$ or $11\left(mod\,12\right)$,
yielding $M$ to be congruent to $0,9,24$ or $33\left(mod\,72\right)$,
and $M$ to be congruent to $1,2$ or $16\left(mod\,24\right)$. These
are called allowed values. Additional congruence conditions were demonstrated
\cite{key-10} on the allowed values of $M$ using Beeckmans' necessary
conditions. Furthermore, it was shown also \cite{key-10-1} that if
$M$ is a square itself, $M$ must be congruent to $1\left(mod\,24\right)$
and $\left(M-1\right)/24$ are all pentagonal numbers, except the
first two. The values of $M$ yielding integer solutions are given
in \cite{key-1-14}.

In this paper, firstly for non-square integer values of $M$, the
Diophantine quadratic equation expressing the sum of consecutive squared
integers equaling a squared integer is transformed into a generalized
Pell equation for which, depending on its number of fundamental solutions,
one or several infinite branch(es) of solutions in $a$ and $s$ are
found analytically, using Chebyshev polynomials. Secondly, for square
values of $M$, the quadratic equation reduces to a difference of
squares for which a finite number of solutions in $a$ and $s$ are
found analytically.

\section{Simple and generalized Pell equations}

Pell equations of the general form

\begin{equation}
X^{2}-DY^{2}=N\label{eq:3-1-1}
\end{equation}
with $X,Y,N\in\mathbb{Z}$ and squarefree $D\in\mathbb{Z}^{+}$, i.e.
$\sqrt{D}\notin\mathbb{Z}^{+}$, have been investigated in various
forms since long (see historical accounts in \cite{key-1-13,key-2,key-5,key-8})
and are treated in several classical text books (see e.g. \cite{key-7,key-3,key-1}
and references therein). A simple reminder is given here and further
details can be found in the references.

For $N=1$, the simple Pell equation reads classically 

\begin{equation}
X^{2}-DY^{2}=1\label{eq:3-1}
\end{equation}
which has, beside the trivial solution $(X_{t},Y_{t})=(1,0)$, a whole
infinite branch of solutions for $k\in\mathbb{Z}^{+}$ given by
\begin{eqnarray}
X_{k} & = & \frac{\left(X_{1}+\sqrt{D}Y_{1}\right)^{k}+\left(X_{1}-\sqrt{D}Y_{1}\right)^{k}}{2}\label{eq:3}\\
Y_{k} & = & \frac{\left(X_{1}+\sqrt{D}Y_{1}\right)^{k}-\left(X_{1}-\sqrt{D}Y_{1}\right)^{k}}{2\sqrt{D}}\label{eq:4}
\end{eqnarray}
where $(X_{1},Y_{1})$ is the fundamental solution to (\ref{eq:3-1}),
i.e. the smallest integer solution ($X_{1}>1,Y_{1}>0,\in\mathbb{Z}^{+}$)
different from the trivial solution. Among the five methods listed
by Robertson \cite{key-9} to find the fundamental solution $(X_{1},Y_{1})$,
the classical method introduced by Lagrange \cite{key-10-1-1}, based
on the continued fraction expansion of the quadratic irrational $\sqrt{D}$,
is central to several other methods.

\noindent For $N=n^{2}$ an integer square, the generalized Pell equation
(\ref{eq:3-1-1}) admits always integer solutions. The variable change
$\left(X^{\prime},Y^{\prime}\right)=\left(\left(X/n\right),\left(Y/n\right)\right)$
transforms the generalized Pell equation in a simple Pell equation
$X{}^{\prime2}-DY{}^{\prime2}=1$ which has integer solutions $\left(X_{k}^{\prime},Y_{k}^{\prime}\right)$.
The integer solutions to the generalized Pell equation can then be
found as $\left(X_{k},Y_{k}\right)=\left(nX_{k}^{\prime},nY_{k}^{\prime}\right)$.
Note however that not all solutions in $\left(X,Y\right)$ may be
found in this way (see e.g. \cite{key-1}).

\noindent For the case where $N$ is not an integer square, the generalized
Pell equation (\ref{eq:3-1-1}) can have either no solution at all,
or one or several fundamental solutions $\left(X_{1},Y_{1}\right)$,
and all integer solutions, if they exist, can be expressed in function
of the fundamental solution(s) $\left(X_{1},Y_{1}\right)$. Several
authors (see e.g. \cite{key-10-1-1,key-1-11,key-7,key-17,key-9,key-1-13,key-15,key-16}
and references therein) discussed how to find the fundamental solution(s)
of the generalized Pell equation, based on Lagrange's method of continued
fractions with various modifications (see e.g. \cite{key-1-2}), and
further how to find additional solutions from the fundamental solution(s). 

\noindent Noting now $\left(x_{f},y_{f}\right)$ the fundamental solutions
of the related simple Pell equation (\ref{eq:3-1}), the other solutions
$\left(X_{k},Y_{k}\right)$ can be found from the fundamental solution(s)
$\left(X_{1},Y_{1}\right)$ by 
\begin{equation}
X_{k}+\sqrt{D}Y_{k}=\pm\left(X_{1}+\sqrt{D}Y_{1}\right)\left(x_{f}+\sqrt{D}y_{f}\right)^{k}\label{eq:40-2}
\end{equation}
for a proper choice of sign $\pm$ \cite{key-9}.

It is less known that Chebyshev polynomials can be used to find the
additional solutions of the generalized Pell equation once the fundamental
solutions $\left(X_{1},Y_{1}\right)$ have been found. In fact, Chebyshev
polynomials $T_{k}\left(x\right)$ and $U_{k}\left(x\right)$ of the
first and second kinds \cite{key-1-1,key-1-3} can be defined as solutions
of the simple Pell equation

\begin{equation}
T_{k}\left(x\right)^{2}-\left(x^{2}-1\right)U_{k-1}\left(x\right)^{2}=1\label{eq:21}
\end{equation}
on a ring $R\left(x\right)$ \cite{key-1-9,key-1-10}. The following
lemma shows how to find the additional solutions of the generalized
Pell equation.
\begin{lem}
For $X,Y,D,N,k\in\mathbb{Z}^{+}$ and $D$ not a perfect square (i.e.
$\sqrt{D}\notin\mathbb{Z}$), if the generalized Pell equation
\begin{equation}
X^{2}-DY^{2}=N\label{eq:57-2}
\end{equation}
admits one or several fundamental solution(s) \textup{$\left(X_{1},Y_{1}\right)$},
then it admits one or several infinite branch(es) of solutions and
these can be written as
\begin{eqnarray}
X_{k} & = & X_{1}T_{k-1}\left(x_{f}\right)+DY_{1}y_{f}U_{k-2}\left(x_{f}\right)\label{eq:57-3}\\
Y_{k} & = & X_{1}y_{f}U_{k-2}\left(x_{f}\right)+Y_{1}T_{k-1}\left(x_{f}\right)\label{eq:58-2}
\end{eqnarray}
in function of the fundamental solution(s) \textup{$\left(X_{1},Y_{1}\right)$}
and of Chebyshev polynomials of the first and second kinds,\textup{
$T_{k-1}\left(x_{f}\right)$} and $U_{k-2}\left(x_{f}\right)$ evaluated
at the fundamental solution \textup{$\left(x_{f},y_{f}\right)$} of
the related simple Pell equation $X^{2}-DY^{2}=1$.\end{lem}
\begin{proof}
For $X,Y,D,N,k,i\in\mathbb{Z}^{+}$ and square free $D$, let $\left(X_{1},Y_{1}\right)$
be one of the fundamental solutions of (\ref{eq:57-2}) if they exist,
and let $\left(x_{f},y_{f}\right)$ be the fundamental solution of
the related simple Pell equation $X^{2}-DY^{2}=1$ (i.e. $x_{f}>1,y_{f}>0$). 

(i) Additional solutions $\left(X_{k},Y_{k}\right)$ of (\ref{eq:57-2})
can then be found by the recurrence relations
\begin{eqnarray}
X_{k} & = & x_{f}X_{k-1}+Dy_{f}Y_{k-1}\label{eq:55-3}\\
Y_{k} & = & x_{f}Y_{k-1}+y_{f}X_{k-1}\label{eq:56-4}
\end{eqnarray}
which can be demonstrated by induction. 

For $k=2$, as $\left(X_{1},Y_{1}\right)$ is a fundamental solution
of (\ref{eq:57-2}), $\left(X_{2},Y_{2}\right)$ obtained from (\ref{eq:55-3})
and (\ref{eq:56-4}) verify also (\ref{eq:57-2}) as $x_{f}^{2}-Dy_{f}^{2}=1$. 

Let $\left(X_{k-1},Y_{k-1}\right)$ be a solution of (\ref{eq:57-2}),
i.e. $X_{k-1}^{2}-DY_{k-1}^{2}=N$. Then multiplying the two terms
on the left of this equation by $1=x_{f}^{2}-Dy_{f}^{2}$, adding
and subtracting $2Dx_{f}y_{f}X_{k-1}Y_{k-1}$, rearranging and replacing
by (\ref{eq:55-3}) and (\ref{eq:56-4}) yield $X_{k}^{2}-DY_{k}^{2}=N$,
i.e. $\left(X_{k},Y_{k}\right)$ is also a solution of (\ref{eq:57-2}).

(ii) Further, to express $X_{k}$ and $Y_{k}$ in function of $X_{1}$,
$Y_{1}$, $x_{f}$ and $y_{f}$ only, one replaces successively for
$3\leq i\leq k$, $X_{i-1}$ and $Y_{i-1}$ in function of $X_{1}$
and $Y_{1}$ in the expressions (\ref{eq:55-3}) and (\ref{eq:56-4})
of $X_{i}$, $Y_{i}$ (with the substitution $x_{f}^{2}+Dy_{f}^{2}=2x_{f}^{2}-1$
whenever needed) to obtain successively Chebyshev polynomials of the
first and second kinds evaluated at $x_{f}$ and of increasing indices,
respectively $i-1$ and $i-2$, i.e. $T_{i-1}\left(x_{f}\right)$
and $U_{i-2}\left(x_{f}\right)$, yielding eventually (\ref{eq:57-3})
and (\ref{eq:58-2}). 

One can verify by induction that (\ref{eq:57-3}) and (\ref{eq:58-2})
yield all solutions to (\ref{eq:57-2}). 

As $\left(X_{1},Y_{1}\right)$ is a fundamental solution of (\ref{eq:57-2}),
for $k=2$, one has $T_{1}\left(x_{f}\right)=x_{f}$ and $U_{0}\left(x_{f}\right)=1$
in (\ref{eq:57-3}) and (\ref{eq:58-2}), yielding directly (\ref{eq:55-3})
and (\ref{eq:56-4}). 

Further, let us assume that $\left(X_{k-1},Y_{k-1}\right)$ with
\begin{eqnarray}
X_{k-1} & = & X_{1}T_{k-2}\left(x_{f}\right)+DY_{1}y_{f}U_{k-3}\left(x_{f}\right)\label{eq:31}\\
Y_{k-1} & = & X_{1}y_{f}U_{k-3}\left(x_{f}\right)+Y_{1}T_{k-2}\left(x_{f}\right)\label{eq:32}
\end{eqnarray}
are a solution of (\ref{eq:57-2}); then replacing (\ref{eq:31})
and (\ref{eq:32}) in (\ref{eq:55-3}) and (\ref{eq:56-4}) yield

\begin{eqnarray}
X_{k} & = & x_{f}\left[X_{1}T_{k-2}\left(x_{f}\right)+DY_{1}y_{f}U_{k-3}\left(x_{f}\right)\right]+\nonumber \\
 &  & Dy_{f}\left[X_{1}y_{f}U_{k-3}\left(x_{f}\right)+Y_{1}T_{k-2}\left(x_{f}\right)\right]\nonumber \\
 & = & X_{1}\left[x_{f}T_{k-2}\left(x_{f}\right)+\left(x_{f}^{2}-1\right)U_{k-3}\left(x_{f}\right)\right]+\nonumber \\
 &  & DY_{1}y_{f}\left[x_{f}U_{k-3}\left(x_{f}\right)+T_{k-2}\left(x_{f}\right)\right]\label{eq:27-1}\\
Y_{k} & = & x_{f}\left[X_{1}y_{f}U_{k-3}\left(x_{f}\right)+Y_{1}T_{k-2}\left(x_{f}\right)\right]+\nonumber \\
 &  & y_{f}\left[X_{1}T_{k-2}\left(x_{f}\right)+DY_{1}y_{f}U_{k-3}\left(x_{f}\right)\right]\nonumber \\
 & = & X_{1}y_{f}\left[x_{f}U_{k-3}\left(x_{f}\right)+T_{k-2}\left(x_{f}\right)\right]+\nonumber \\
 &  & Y_{1}\left[x_{f}T_{k-2}\left(x_{f}\right)+\left(x_{f}^{2}-1\right)U_{k-3}\left(x_{f}\right)\right]\label{eq:28}
\end{eqnarray}
where $Dy_{f}^{2}$ has been replaced by $Dy_{f}^{2}=x_{f}^{2}-1$
in (\ref{eq:27-1}) and (\ref{eq:28}). As
\begin{eqnarray}
T_{k-1}\left(x_{f}\right) & = & x_{f}T_{k-2}\left(x_{f}\right)+\left(x_{f}^{2}-1\right)U_{k-3}\left(x_{f}\right)\label{eq:35}\\
U_{k-2}\left(x_{f}\right) & = & x_{f}U_{k-3}\left(x_{f}\right)+T_{k-2}\left(x_{f}\right)\label{eq:36}
\end{eqnarray}

(see e.g. \cite{key-1-1}), (\ref{eq:27-1}) and (\ref{eq:28}) yield
directly (\ref{eq:57-3}) and (\ref{eq:58-2}). Replacing now (\ref{eq:57-3})
and (\ref{eq:58-2}) in (\ref{eq:57-2}) gives

\begin{equation}
X_{k}^{2}-DY_{k}^{2}=\left(X_{1}^{2}-DY_{1}^{2}\right)\left(T_{k-1}\left(x_{f}\right)^{2}-Dy_{f}^{2}U_{k-2}\left(x_{f}\right)^{2}\right)=N\label{eq:29}
\end{equation}
by (\ref{eq:21}) with $Dy_{f}^{2}=x_{f}^{2}-1$, showing that $\left(X_{k},Y_{k}\right)$
(\ref{eq:57-3}, \ref{eq:58-2}) also solve (\ref{eq:57-2}).

Finally, as $k$ is unbound, there is an infinity of solutions (\ref{eq:57-3})
and (\ref{eq:58-2}).
\end{proof}

\section{General method to find all solutions}

The sum of $M>1$ consecutive integer squares starting from $a^{2}\geq1$
being equal to an integer square $s^{2}$ can be written in all generality
as \cite{key-10-1}

\begin{equation}
\sum_{i=0}^{M-1}\left(a+i\right)^{2}=M\left[\left(a+\frac{M-1}{2}\right)^{2}+\frac{M^{2}-1}{12}\right]=s^{2}\label{eq:53-3}
\end{equation}

where $M$ are allowed values (see \cite{key-10-1,key-10}). To find
all integer solutions of (\ref{eq:53-3}), two cases are considered
and treated separately: first, $M$ is not a squared integer, and
second, $M$ is a squared integer.

\subsection{$M$ not a squared integer}

The next theorem allows to find all the solutions to (\ref{eq:53-3})
in $a$ and $s$ for allowed values of $M$ not being squared integers.
\begin{thm}
For $M>1,\sigma,j,k,a_{k,j},s_{k,j},x_{f},y_{f}\in\mathbb{Z}^{+},\lambda\in\mathbb{Q}$,
for all allowed square free values of $M$ (i.e. $\sqrt{M}\notin\mathbb{Z}$),
there is a number $\sigma\geq1$ of infinite branch(es) of values
of $a_{k,j}$, $1\leq j\leq\sigma$, such that the sums of squares
of $M$ consecutive integers starting from $a_{k,j}$ are equal to
squared positive integers $s_{k,j}^{2}$ and these can be written
in function of Chebyshev polynomials of the first and second kinds,\textup{
$T_{k-1}\left(x_{f}\right)$} and $U_{k-2}\left(x_{f}\right)$ as\textup{
\begin{eqnarray}
a_{k,j} & = & \frac{2\lambda s_{1,j}y_{f}U_{k-2}\left(x_{f}\right)+\left(2a_{1,j}+M-1\right)T_{k-1}\left(x_{f}\right)-\left(M-1\right)}{2}\label{eq:26}\\
s_{k,j} & = & s_{1,j}T_{k-1}\left(x_{f}\right)+\frac{\lambda M}{2}y_{f}\left(2a_{1,j}+M-1\right)U_{k-2}\left(x_{f}\right)\label{eq:27}
\end{eqnarray}
}with $\lambda=1$ for $M\equiv1\left(mod\,2\right)$ or $M\equiv2\left(mod\,4\right)$,
and $\lambda=1/2$ for $M\equiv0\left(mod\,4\right)$, and where $\left(a_{1,j},s_{1,j}\right)$
are the smallest positive values of $\left(a_{k,j},s_{k,j}\right)$
solutions of (\ref{eq:53-3}) and \textup{$\left(x_{f},y_{f}\right)$}
is the fundamental solution of the simple Pell equation $X^{2}-\left(\lambda^{2}M\right)Y^{2}=1$.\end{thm}
\begin{proof}
For $M>1,\sigma,j,k,a,s,a_{k,j},s_{k,j},x_{f},y_{f},X,Y,N,D\in\mathbb{Z}^{+},\lambda\in\mathbb{Q}$,
for the allowed square free values of $M$, rewriting (\ref{eq:53-3})
for $M\equiv1\left(mod\,2\right)$ as

\begin{equation}
s^{2}-M\left(a+\frac{M-1}{2}\right)^{2}=\frac{M\left(M^{2}-1\right)}{12}\label{eq:55-2}
\end{equation}
or for $M\equiv0\left(mod\,4\right)$ as 
\begin{equation}
s^{2}-\frac{M}{4}\left(2a+M-1\right)^{2}=\frac{M\left(M^{2}-1\right)}{12}\label{eq:56-3}
\end{equation}
or for $M\equiv2\left(mod\,4\right)$ as

\begin{equation}
\left(2s\right)^{2}-M\left(2a+M-1\right)^{2}=\frac{M\left(M^{2}-1\right)}{3}\label{eq:62-1}
\end{equation}
transform (\ref{eq:53-3}) in generalized Pell equations (\ref{eq:3-1-1})
in $X=s$ or $2s$ and $Y=\left(a+\left(M-1\right)/2\right)$ or $\left(2a+M-1\right)$,
with $N=M\left(M^{2}-1\right)/12$ or $M\left(M^{2}-1\right)/3$ and
$D=M$ or $M/4$. 

If these generalized Pell equations (\ref{eq:55-2}) to (\ref{eq:62-1})
admit $\sigma$ solution(s), then for $1\leq j\leq\sigma$,

(i) for $M\equiv1\left(mod\,2\right)$, let $\left(s_{1,j},\left(a_{1,j}+\left(M-1\right)/2\right)\right)$
be the $j^{th}$ fundamental solution of (\ref{eq:55-2}) and let
$\left(x_{f},y_{f}\right)$ be the fundamental solution of the related
simple Pell equation $X^{2}-MY^{2}=1$, i.e. $x_{f}>1$ and $y_{f}>0$.
Then, (\ref{eq:57-3}) and (\ref{eq:58-2}) yield
\begin{eqnarray}
a_{k,j} & = & s_{1,j}y_{f}U_{k-2}\left(x_{f}\right)+\left(a_{1,j}+\frac{M-1}{2}\right)T_{k-1}\left(x_{f}\right)-\left(\frac{M-1}{2}\right)\label{eq:63}\\
s_{k,j} & = & s_{1,j}T_{k-1}\left(x_{f}\right)+My_{f}\left(a_{1,j}+\frac{M-1}{2}\right)U_{k-2}\left(x_{f}\right)\label{eq:62}
\end{eqnarray}

(ii) for $M\equiv0\left(mod\,4\right)$, similarly let $\left(s_{1,j},\left(2a_{1,j}+M-1\right)\right)$
be the $j^{th}$ fundamental solution of (\ref{eq:56-3}) and let
$\left(x_{f},y_{f}\right)$ be the fundamental solution of the related
simple Pell equation $X^{2}-\left(M/4\right)Y^{2}=1$. Then, (\ref{eq:57-3})
and (\ref{eq:58-2}) yield
\begin{eqnarray}
a_{k,j} & = & \frac{s_{1,j}y_{f}U_{k-2}\left(x_{f}\right)+\left(2a_{1,j}+M-1\right)T_{k-1}\left(x_{f}\right)-\left(M-1\right)}{2}\label{eq:65}\\
s_{k,j} & = & s_{1,j}T_{k-1}\left(x_{f}\right)+\frac{M}{4}y_{f}\left(2a_{1,j}+M-1\right)U_{k-2}\left(x_{f}\right)\label{eq:64}
\end{eqnarray}

(iii) for $M\equiv2\left(mod\,4\right)$, similarly let $\left(2s_{1,j},\left(2a_{1,j}+M-1\right)\right)$
be the $j^{th}$ fundamental solution of (\ref{eq:62-1}) and let
$\left(x_{f},y_{f}\right)$ be the fundamental solution of the related
simple Pell equation $X^{2}-MY^{2}=1$. Then, (\ref{eq:57-3}) and
(\ref{eq:58-2}) yield
\begin{eqnarray}
a_{k,j} & = & \frac{2s_{1,j}y_{f}U_{k-2}\left(x_{f}\right)+\left(2a_{1,j}+M-1\right)T_{k-1}\left(x_{f}\right)-\left(M-1\right)}{2}\label{eq:68-1}\\
s_{k,j} & = & s_{1,j}T_{k-1}\left(x_{f}\right)+\frac{M}{2}y_{f}\left(2a_{1,j}+M-1\right)U_{k-2}\left(x_{f}\right)\label{eq:67-1}
\end{eqnarray}

Finally, as $k$ is unbound, there is in each case and for each $1\leq j\leq\sigma$
an infinity of solutions $\left(s_{k,j},a_{k,j}\right)$.
\end{proof}
Note that some of the first solutions $a_{1,j}$ may be rejected if
the $j^{th}$ fundamental solution of (\ref{eq:55-2}) (or (\ref{eq:56-3})
or (\ref{eq:62-1})) is such that $\left(a_{1,j}+\left(M-1\right)/2\right)<\left(M-1\right)/2$,
yielding a non-positive value of $a_{1,j}$.

In the following examples, the method indicated by Matthews \cite{key-16}
based on an algorithm by Frattini \cite{key-2-1,key-2-2,key-2-3}
using Nagell's bounds \cite{key-7,key-4-1} is used to find the fundamental
solution(s) of the generalized Pell equation.

A first example for the case $M\equiv11\left(mod\,12\right)$, let
$M=11$. Then, (\ref{eq:55-2}) reads $s^{2}-11\left(a+5\right)^{2}=110$,
which has $\sigma=2$ fundamental solutions, yielding, with $1\leq j\leq2$
, $\left(s_{1,j},\left(a_{1,j}+5\right)\right)=\left(11,1\right),\left(77,23\right)$
and the fundamental solution of the related simple Pell equation $X^{2}-11Y^{2}=1$
is $\left(x_{f},y_{f}\right)=\left(10,3\right)$. Replacing in (\ref{eq:63})
and (\ref{eq:62}) yield then the solutions given in Table 1. The
first solution $\left(a_{1,1},s_{1,1}\right)$ is rejected as $a_{1,1}<0$.
The solutions are then ordered as $a_{1,2}<a_{2,1}<a_{2,2}<a_{3,1}<...$.

\begin{table}
\noindent \begin{centering}
\caption{First solutions $\left(a_{k,j},s_{k,j}\right)$ for $M=11$, $1\leq j\leq2$
and $1\leq k\leq6$ of the $\sigma=2$ infinite branches of solutions
of $s^{2}-11\left(a+5\right)^{2}=110$ }

\par\end{centering}

\begin{centering}
\begin{tabular}{|c|c|c|c|c|}
\hline 
$k$ & $a_{k,1}$ & $s_{k,1}$ & $a_{k,2}$ & $s_{k,2}$\tabularnewline
\hline 
\hline 
1 & {[}-4{]} & {[}11{]} & 18 & 77\tabularnewline
\hline 
2 & 38 & 143 & 456 & 1529\tabularnewline
\hline 
3 & 854 & 2849 & 9192 & 30503\tabularnewline
\hline 
4 & 17132 & 56837 & 183474 & 608531\tabularnewline
\hline 
5 & 341876 & 1133891 & 3660378 & 12140117\tabularnewline
\hline 
6 & 6820478 & 22620983 & 73024176 & 242193809\tabularnewline
\hline 
\end{tabular}
\par\end{centering}

\centering{}{[}$a_{1,1}${]}: solution rejected as $a_{1,1}\leq0$
\end{table}

A second example for the case $M\equiv0\left(mod\,24\right)$, let
$M=24$. Then, (\ref{eq:56-3}) reads $s^{2}-6\left(2a+23\right)^{2}=1150$,
having $\sigma=6$ fundamental solutions, $\left(s_{1,j},\left(2a_{1,j}+23\right)\right)$
$=\left(34,1\right),\left(38,7\right),\left(50,15\right)\left(70,25\right),\left(106,41\right),\left(158,63\right)$
and the fundamental solution of the related simple Pell equation $X^{2}-6Y^{2}=1$
is $\left(x_{f},y_{f}\right)=\left(5,2\right)$. Replacing in (\ref{eq:65})
and (\ref{eq:64}) yield then the solutions given in Table 2. The
first three solutions $\left(a_{1,j},s_{1,j}\right)$ for $1\leq j\leq3$
are rejected as $a_{1,j}<0$. The solutions are then ordered as $a_{1,4}<a_{1,5}<a_{1,6}<a_{2,1}<a_{2,3}<a_{2,4}<...$.
Note that the first solution of the fourth branch ($a_{1,4}=1$, $s_{1,4}=70$)
gives the second solution of Lucas' cannonball problem.

\begin{table}
\noindent \begin{centering}
\caption{First solutions $\left(a_{k,j},s_{k,j}\right)$ for $M=24$, $1\leq j\leq6$
and $1\leq k\leq6$ of the $\sigma=6$ infinite branches of solutions
of $s^{2}-6\left(2a+23\right)^{2}=1150$ }

\par\end{centering}

\begin{centering}
\begin{tabular}{|c|c|c|c|c|c|c|}
\hline 
$k$ & $a_{k,1}$ & $s_{k,1}$ & $a_{k,2}$ & $s_{k,2}$ & $a_{k,3}$ & $s_{k,3}$\tabularnewline
\hline 
\hline 
1 & {[}-11{]} & {[}34{]} & {[}-8{]} & {[}38{]} & {[}-4{]} & {[}50{]}\tabularnewline
\hline 
2 & 25  & 182 & 44 & 274 & 76 & 430\tabularnewline
\hline 
3 & 353 & 1786 & 540 & 2702 & 856 & 4250\tabularnewline
\hline 
4 & 3597 & 17678 & 5448 & 26746 & 8576 & 42070\tabularnewline
\hline 
5 & 35709 & 174994 & 54032 & 264758 & 84996 & 416450\tabularnewline
\hline 
6 & 353585 & 1732262 & 534964 & 2620834 & 841476 & 4122430\tabularnewline
\hline 
\end{tabular}
\par\end{centering}

\begin{centering}
\begin{tabular}{|c|c|c|c|c|c|c|}
\hline 
$k$ & $a_{k,4}$ & $s_{k,4}$ & $a_{k,5}$ & $s_{k,5}$ & $a_{k,6}$ & $s_{k,6}$\tabularnewline
\hline 
\hline 
1 & 1 & 70 & 9 & 106 & 20 & 158\tabularnewline
\hline 
2 & 121 & 650 & 197 & 1022 & 304 & 1546\tabularnewline
\hline 
3 & 1301 & 6430 & 2053 & 10114 & 3112 & 15302\tabularnewline
\hline 
4 & 12981 & 63650 & 20425 & 100118 & 30908 & 151474\tabularnewline
\hline 
5 & 128601 & 630070 & 202289 & 991066 & 306060 & 1499438\tabularnewline
\hline 
6 & 1273121 & 6237050 & 2002557 & 9810542 & 29991872 & 146929622\tabularnewline
\hline 
\end{tabular}
\par\end{centering}

\centering{}{[}$a_{1,j}${]}: solutions rejected as $a_{1,j}\leq0$
for $1\leq j\leq3$
\end{table}

A third example for the case $M\equiv2\left(mod\,24\right)$, let
$M=2$. Then, (\ref{eq:62-1}) reads $\left(2s\right)^{2}-2\left(2a+1\right)^{2}=2$,
having $\sigma=1$ fundamental solution $\left(2s_{1,j},\left(2a_{1,j}+1\right)\right)=\left(2,1\right)$
and the fundamental solution of the related simple Pell equation $X^{2}-2Y^{2}=1$
is $\left(x_{f},y_{f}\right)=\left(3,2\right)$. Replacing in (\ref{eq:68-1})
and (\ref{eq:67-1}) yield then the solutions given in Table 3, where
the first solution is again to be rejected (it corresponds to the
identity relation $0^{2}+1^{2}=1^{2}$) and the second solution is
the Pythagorean relation $3^{2}+4^{2}=5^{2}$.

\begin{table}
\noindent \begin{centering}
\caption{First solutions $\left(a_{k,1},s_{k,1}\right)$ for $M=2$ and $1\leq k\leq6$
of the single infinite branch of solutions of $\left(2s\right)^{2}-2\left(2a+1\right)^{2}=2$ }

\par\end{centering}

\begin{centering}
\begin{tabular}{|c|c|c|}
\hline 
$k$ & $a_{k,1}$ & $s_{k,1}$\tabularnewline
\hline 
\hline 
1 & {[}0{]} & {[}1{]}\tabularnewline
\hline 
2 & 3  & 5\tabularnewline
\hline 
3 & 20 & 29\tabularnewline
\hline 
4 & 119 & 169\tabularnewline
\hline 
5 & 696 & 985\tabularnewline
\hline 
6 & 4059 & 5741\tabularnewline
\hline 
\end{tabular}
\par\end{centering}

\centering{}{[}$a_{1,1}${]}: solution rejected as $a_{1,1}\leq0$
\end{table}

Still for the case $M=2\left(mod\,24\right)$, let $M=842$ which
does not yield solutions to (\ref{eq:53-3}). Indeed, although the
related simple Pell equation $X^{2}-842Y^{2}=1$ has the fundamental
solution $\left(x_{f},y_{f}\right)=\left(1683,58\right)$, the generalized
Pell equation from (\ref{eq:62-1}) $\left(2s\right)^{2}-842\left(2a+841\right)^{2}=198982282$
has no fundamental solution ($\sigma=0$). This case was already signaled
by Beeckmans \cite{key-1-4}: the value of $M=842=24\times35+2$,
although complying with Beeckmans' conditions does not yield solutions
to (\ref{eq:53-3}) (see also \cite{key-10}).

\subsection{$M$ is a squared integer}

It was demonstrated \cite{key-10-1} that, if $M$ is a square integer,
then for the sums of $M$ consecutive squared integers to equal integer
squares, $M\equiv1\left(mod\,24\right)$ and $\exists g_{i}\in\mathbb{Z}^{+}$
such that $M=24g_{n}+1$ where $g_{n}=n\left(3n-1\right)/2$ are all
generalized pentagonal numbers $\forall n\in\mathbb{Z}$ \cite{key-8-1,key-6},
yielding $M=\left(6n-1\right)^{2}$, i.e. $g_{n}=0,1,2,5,7,12,15,22,26,35,40,51,57,...$
\cite{key-5-1}, yielding 

$M=1,25,49,121,169,289,361,529,625,841,961,1225,1369,...$ of which
the first two $M=1,25$, should be rejected as $M>1$ and $a>0$ (see
further).

For $M$ an integer square, the above method with solutions of the
Pell equation can clearly not be followed as Pell equations are not
defined for $D=M$ being a squared integer. Instead, another method
(see e.g. \cite{key-1-11} p. 486, and \cite{key-17}) is used in
the following theorem showing how to find the finite number of solutions
for $M$ being a squared integer. 
\begin{thm}
For $M>1,\varphi,k,a_{k},s_{k}\in\mathbb{Z}^{+}$, $n\in\mathbb{Z}$,
for all allowed squared integer values of $M=\left(6n-1\right)^{2}$,
there is a finite number $\varphi$ of values of $a_{k}$ such that
the sums of squares of $M$ consecutive integers starting from $a_{k}$
are equal to squared positive integers $s_{k}^{2}$, that can be written
as 
\begin{eqnarray}
s_{k} & = & \left(6n-1\right)\left(\frac{v_{k}+u_{k}}{2}\right)\label{eq:37}\\
a_{k} & = & \frac{v_{k}-u_{k}}{2}-6n\left(3n-1\right)\label{eq:38}
\end{eqnarray}

where $u_{k}$ and $v_{k}$ are the factor and co-factor of\textup{
$\left[2n\left(3n-1\right)\left(6n\left(3n-1\right)+1\right)\right]$,
}with $u_{k}<v_{k}$, $u_{k}\equiv v_{k}\equiv0\left(mod\,2\right)$
and $1\leq k\leq\varphi$.\end{thm}
\begin{proof}
For $M>1,\varphi,k,a,s,a_{k},s_{k}\in\mathbb{Z}^{+}$, $n\in\mathbb{Z}$,
from (\ref{eq:53-3}), $s$ must be such as $s\equiv0\left(mod\,\left(6n-1\right)\right)$.
Replacing in (\ref{eq:55-2}) yields then 
\begin{equation}
\left(\frac{s}{6n-1}\right)^{2}-\left(a+6n\left(3n-1\right)\right)^{2}=2n\left(3n-1\right)\left(6n\left(3n-1\right)+1\right)\label{eq:70}
\end{equation}
i.e. the difference of two integer squares must be an even integer. 

One has then to determine all the integer values of $X_{k}$ and $Y_{k}$
solutions of the equation $X^{2}-Y^{2}=N$, with $X=s/\left(6n-1\right)$,
$Y=\left(a+6n\left(3n-1\right)\right)$ and $N=2n\left(3n-1\right)\left(6n\left(3n-1\right)+1\right)$.
For this, let $N=u_{k}v_{k}$ and only both even factor and co-factor
$u_{k}$ and $v_{k}$ are considered as $N\equiv0\left(mod\,4\right)$
\cite{key-17}. As $N$ is finite, there is a finite number $\varphi$
of ways of decomposing $N$ in product of two even factors. Then,
with $u_{k}<v_{k}$ and $1\leq k\leq\varphi$, $X_{k}=\left(v_{k}+u_{k}\right)/2$
and $Y_{k}=\left(v_{k}-u_{k}\right)/2$, yielding $s_{k}=\left(6n-1\right)\left(v_{k}+u_{k}\right)/2$
and $a_{k}=\left(\left(v_{k}-u_{k}\right)/2\right)-6n\left(3n-1\right)$
.
\end{proof}
Note that here also some of the first solutions $a_{k}$ may be rejected
if half the difference of the factor and co factor of $N$ is such
that $\left(\left(v_{k}-u_{k}\right)/2\right)<6n\left(3n-1\right)$,
yielding a non-positive value of $a_{k}$. 

As a first example, let $M=25$ with $n=1$. Then $X=s/5$, $Y=a+12$
and there is only one way to decompose $N=52$ in the product of two
even integer factors, $N=52=2\times26=u_{1}v_{1}$, yielding then
$\varphi=1$ and there is only one solution, given by $X_{1}=14$
and $Y_{1}=12$, or $s_{1}=70$ and $a_{1}=0$. This case for $M=25$
must be rejected as it has no solution with $a>0$. Note however that
this solution with $s=70$ and $a=0$ for the case $M=25$ is obviously
equivalent to the solution with $s=70$ and $a=1$ for the case $M=24$
of Lucas' cannonball problem.

A second example, let $M=289$ with $n=3$. Then $X=s/17$, $Y=a+144$
and $N=6960$. As there are twelve ways to decompose $N=6960$ in
products of two even integer factors, there are $\varphi=12$ solutions
in $X$ and $Y$ given in Table 4, five of which have to be rejected
as the corresponding values of $a_{k}$ are negative.

\begin{table}
\caption{All $\varphi=12$ solutions for $M=289$ with $N=u_{k}v_{k}=6960$ }

\begin{centering}
\begin{tabular}{|c|c|c|c|c|c|}
\hline 
$k$ & $u_{k}\times v_{k}$ & $X_{k}$ & $Y_{k}$ & $s_{k}$ & $a_{k}$\tabularnewline
\hline 
\hline 
1 & $60\times116$ & 88 & 28 & {[}1496{]} & {[}-116{]}\tabularnewline
\hline 
2 & $58\times120$ & 89 & 31 & {[}1513{]} & {[}-113{]}\tabularnewline
\hline 
3 & $40\times174$ & 107 & 67 & {[}1819{]} & {[}-77{]}\tabularnewline
\hline 
4 & $30\times232$ & 131 & 101 & {[}2227{]} & {[}-43{]}\tabularnewline
\hline 
5 & $24\times290$ & 157 & 133 & {[}2669{]} & {[}-11{]}\tabularnewline
\hline 
6 & $20\times348$ & 184 & 164 & 3128 & 20\tabularnewline
\hline 
7 & $12\times580$ & 296 & 284 & 5032 & 140\tabularnewline
\hline 
8 & $10\times696$ & 353 & 343 & 6001 & 199\tabularnewline
\hline 
9 & $8\times870$ & 439 & 431 & 7463 & 287\tabularnewline
\hline 
10 & $6\times1160$ & 583 & 577 & 9911 & 433\tabularnewline
\hline 
11 & $4\times1740$ & 872 & 868 & 14824 & 724\tabularnewline
\hline 
12 & $2\times3480$ & 1741 & 1739 & 29597 & 1595\tabularnewline
\hline 
\end{tabular}
\par\end{centering}

\centering{}{[}$a_{k}${]}: solutions to be rejected as $a_{k}<0$
\end{table}

\section{Conclusion}

The problem of finding all the integer solutions of the sum of $M$
consecutive integer squares starting at $a^{2}\geq1$ being equal
to a squared integer $s^{2}$ can be written as a Diophantine quadratic
equation $M\left[\left(a+\left(M-1\right)/2\right)^{2}+\left(M^{2}-1\right)/12\right]=s^{2}$
in variables $a$ and $s$. Based on previous results, it is known
that integer solutions exist only if $M\equiv0,9,24$ or $33\left(mod\,72\right)$;
or $M\equiv1,2$ or $16\left(mod\,24\right)$; or $M\equiv11\left(mod\,12\right)$. 

If $M$ is different from a square integer, the Diophantine quadratic
equation is solved generally by transforming it into a generalized
Pell equation whose form depends on the $\left(mod\,4\right)$ congruent
value of $M$, and whose solutions, if existing, yield all the solutions
in $a$ and $s$ for a given value of $M$. Depending on whether this
generalized Pell equation admits one or several fundamental solution(s),
there are one or several infinite branches of solutions in $a$ and
$s$ that can be written simply in function of Chebyshev polynomials
evaluated at the fundamental solutions of the related simple Pell
equation. 

If $M$ is a square integer, for $M\equiv1\left(mod\,24\right)$ and
$M=\left(6n-1\right)^{2}$, $\forall n\in\mathbb{Z}$, then the Diophantine
quadratic equation in variables $a$ and $s$ reduces to a simple
difference of integer squares which admits a finite number of solutions,
yielding a finite number solutions in $a$ and $s$ to the initial
problem.

\section{Acknowledgment}

The author acknowledges Dr C. Thiel for the help brought throughout
this paper.

\end{document}